\documentclass[reqno,a4paper,12pt]{amsart}
\usepackage[all,poly]{xy}
\usepackage{amsfonts}
\usepackage[mathcal]{eucal}
\usepackage{amssymb}
\usepackage{amsmath}
\usepackage{mathrsfs}
\usepackage{amsthm}
\usepackage{color}
\usepackage[pagebackref,colorlinks]{hyperref}
\usepackage{enumerate}

\theoremstyle{plain}
\newtheorem{lemma}{Lemma}[section] 
\newtheorem{theorem}[lemma]{Theorem}
\newtheorem{corollary}[lemma]{Corollary}
\newtheorem{proposition}[lemma]{Proposition}

\theoremstyle{definition}

\newtheorem{example}[lemma]{Example}

\newcommand{\M}{\operatorname{\mathbb M}}
\newcommand{\gr}{\operatorname{gr}}
\newcommand{\V}{\mathcal V}
\newcommand{\Stab}{\operatorname{Stab}}

\newcommand{\Zset}{\mathbb Z}

\newcommand{\so}{\mathbf{s}}
\newcommand{\ra}{\mathbf{r}}

\newcommand{\Aut}{\operatorname{Aut}}
\newcommand{\AUT}{\operatorname{AUT}}
\newcommand{\End}{\operatorname{End}}
\newcommand{\END}{\operatorname{END}}
\newcommand{\Hom}{\operatorname{Hom}}
\newcommand{\HOM}{\operatorname{HOM}}

\title[Crossed product Leavitt path algebras]{Crossed product Leavitt path algebras}

\author{Roozbeh Hazrat}
\address{Centre for Research in Mathematics and Data Science, Western Sydney University, Australia} 
\email{r.hazrat@westernsydney.edu.au}

\author{Lia Va\v s}
\address{Department of Mathematics, Physics and Statistics, University of the Sciences, Philadelphia, PA 19104, USA}
\email{l.vas@usciences.edu}

\subjclass[2010]{ 
16S88 % LPAs. Ovo nije u 2010 classification ali ne znam kako da napravim da je 2020 jer je subjclass komanda definisana samo za 2010
16W50 % graded rings 
16S35 % Twisted and skew group rings, crossed products
16S50 % Endomorphism rings; matrix rings
19A49% $K_0$ of other rings. 19A is Grothendieck groups and $K_0$
}

\thanks{The first author would like to acknowledge Australian Research Council grant DP160101481. }

\begin{document}
 
\begin{abstract} 
If $E$ is a directed graph and $K$ is a field, the Leavitt path algebra $L_K(E)$ of $E$ over $K$ is naturally graded by the group of integers $\mathbb Z.$ We formulate properties of the graph $E$ which are equivalent with $L_K(E)$ being a crossed product, a skew group ring, or a group ring with respect to this natural grading. We state this main result so that the algebra properties of $L_K(E)$ are also characterized in terms of the pre-ordered group properties of the Grothendieck $\mathbb Z$-group of $L_K(E)$. If $E$ has finitely many vertices, we characterize when $L_K(E)$ is strongly graded in terms of the properties of $K_0^\Gamma(L_K(E)).$ Our proof also provides an alternative to the known proof of the equivalence $L_K(E)$ is strongly graded if and only if $E$ has no sinks for a finite graph $E.$ We also show that, if unital, the algebra $L_K(E)$ is strongly graded and graded unit-regular if and only if $L_K(E)$ is a crossed product.

In the process of showing the main result, we obtain conditions on a group $\Gamma$ and a $\Gamma$-graded division ring $K$ equivalent with the requirements that a $\Gamma$-graded matrix ring $R$ over $K$ is strongly graded, a crossed product, a skew group ring, or a group ring. We characterize these properties also in terms of the action of the group $\Gamma$ on the Grothendieck $\Gamma$-group $K_0^\Gamma(R).$  
\end{abstract}
 
\maketitle

\section{Introduction}

After emerging about fifteen years ago, Leavitt path algebras became a focus of intense research which inspired various generalizations, surprising connections with other areas of mathematics, and a variety of interesting conjectures. If $K$ is a field considered trivially graded by $\Zset,$ the Leavitt path algebra $L_K(E)$ of a directed graph $E$ is naturally graded by the ring of integers $\Zset.$ One trend in Leavitt path algebra research is to characterize (graded) algebra properties of $L_K(E)$ in terms of the graph properties of $E.$ Thus, if (graded) algebra properties (P1) and (P2) of $L_K(E)$ can be characterized in terms of the properties (G1) and (G2) of $E$ respectively, one can easily create examples of (graded) algebras having (P1) and not (P2) by considering graphs having (G1) and not (G2). 

Strongly graded rings were introduced as generalized crossed product rings by Kanzaki in \cite{Kanzaki}. These rings were later renamed strongly graded rings by Dade who systematically studied them in \cite{Dade}. In the chain of implications below, the relation $X$ $\Rightarrow$ $Y$ denotes that the class of graded rings $X$ is contained in the class of graded rings $Y.$ In addition, in the chain below, whenever we have $X$ $\Rightarrow$ $Y,$ the elements of $X$ are considered to be especially well-behaved in the class $Y.$
\begin{center}
Group rings $\Rightarrow$ skew group rings $\Rightarrow$ crossed products $\Rightarrow$ strongly graded rings $\Rightarrow$ graded rings.
\end{center} 
We review the definitions of all five classes in Section \ref{section_prerequisites}. 

The property of $E$ equivalent with the condition that $L_K(E)$ is strongly graded was produced in \cite{Roozbeh_Israeli} for finite graphs and in \cite{Lisa_et_al_strongly_graded} for general graphs. In the main result of this paper, Theorem \ref{main}, we formulate properties of $E$ equivalent with conditions that $L_K(E)$ is a crossed product, a skew group ring, or a group ring. Hence, Leavitt path algebras in any of the five classes listed above can be characterized by the properties of the underlying graph. We formulate Theorem \ref{main} so that the algebra properties are also related to the appropriate pre-ordered group properties of the Grothendieck $\Zset$-group of $L_K(E)$.

In the process of proving Theorem \ref{main}, we prove Theorem \ref{three_characterizations} listing conditions equivalent with the requirements that a graded matrix ring $R=\M_n(K)(\gamma_1,\ldots, \gamma_n)$ over a $\Gamma$-graded division ring $K$ is strongly graded, a crossed product, a skew group ring or a group ring for any group $\Gamma.$ If $\Gamma_K$ is the set of $\gamma\in\Gamma$ such that the $\gamma$-component of $K$ is nonzero, then $\Gamma_K$ is a subgroup of $\Gamma$ and the permutation module $\Zset[\Gamma/\Gamma_K]$ is naturally isomorphic to the Grothendieck $\Gamma$-group $K_0^\Gamma(R).$ We formulate the conditions of Theorem \ref{three_characterizations} in terms of the properties of the cosets $\gamma_1\Gamma_K,\ldots, \gamma_n\Gamma_K$ as well as in terms of the type of action of $\Gamma$ on the ordered group $\Zset[\Gamma/\Gamma_K].$ 

If the set of vertices $E^0$ is finite, Theorem \ref{strongly_graded} characterizes when $L_K(E)$ is strongly graded in terms of the properties of $K_0^\Gamma(L_K(E)).$ In addition, our proof provides an alternative to the known proof of the equivalence $L_K(E)$ is strongly graded if and only if $E$ has no sinks for a finite graph $E$ (\cite[Theorem 3.15]{Roozbeh_Israeli} and \cite[Proposition 45]{Nystedt_Oinert}). In Corollary \ref{unit_reg_corollary}, we show that, if unital, $L_K(E)$ is strongly graded and graded unit-regular if and only if $L_K(E)$ is a crossed product.

\section{Prerequisites and preliminaries}\label{section_prerequisites}

We use $\Gamma$ to denote an arbitrary group unless otherwise stated. We use multiplicative notation for the operation of $\Gamma$ and $\varepsilon$ for the identity element.

\subsection{Graded rings}\label{subsection_graded_rings}
A ring $R$ is a \emph{$\Gamma$-graded ring} if $R=\bigoplus_{ \gamma \in \Gamma} R_{\gamma}$ for additive subgroups $R_{\gamma}$ of $R$ such that $R_{\gamma}  R_{\delta} \subseteq R_{\gamma\delta}$ for all $\gamma, \delta \in \Gamma$. If it is clear from the context that $R$ is graded by $\Gamma$, $R$ is said to be a graded ring. The elements of $\bigcup_{\gamma \in \Gamma} R_{\gamma}$ are the \emph{homogeneous elements} of $R.$ 
A unital graded ring $R$ is a \emph{graded division ring} if every nonzero homogeneous element has a multiplicative inverse. If a graded division ring $R$ is commutative then $R$ is a {\em graded field}. The set 
\[\Gamma_R=\{\gamma\in\Gamma\,|\, R_\gamma\neq 0\}\]
is the {\em support} of a graded ring $R.$ A graded ring $R$ is \emph{trivially graded} if $\Gamma_R=\{\varepsilon\}.$ If $K$ is a graded division ring, $\Gamma_K$ is a subgroup of $\Gamma.$

We adopt the standard definitions of graded ring homomorphisms and isomorphisms, graded left and right $R$-modules, graded module homomorphisms, graded algebras, graded left and right free and projective modules as defined in \cite{NvO_book} and \cite{Roozbeh_book}.
If $M$ is a graded right $R$-module and $\gamma\in\Gamma,$ the $\gamma$-\emph{shifted or $\gamma$-suspended} graded right $R$-module $(\gamma)M$ is defined as the module $M$ with the $\Gamma$-grading given by $$(\gamma)M_\delta = M_{\gamma\delta}$$ for all $\delta\in \Gamma.$ 
Analogously, if $M$ is a graded left $R$-module, the $\gamma$-shifted  left $R$-module $M(\gamma)$ is the module $M$ with the $\Gamma$-grading given by $M(\gamma)_\delta = M_{\delta\gamma}$ for all $\delta\in \Gamma.$
Any finitely generated graded free right $R$-module is of the form $(\gamma_1)R\oplus\ldots\oplus (\gamma_n)R$ for $\gamma_1, \ldots,\gamma_n\in\Gamma$ and an analogous statement holds for finitely generated graded free left $R$-modules (both \cite{NvO_book} and \cite{Roozbeh_book} contain details).  

If $M$ and $N$ are graded right $R$-modules and $\gamma\in\Gamma$, then $\Hom_R(M,N)_\gamma$ denotes 
$$\Hom_R(M,N)_\gamma=\{f\in \Hom_R(M, N)\,|\,f(M_\delta)\subseteq N_{\gamma\delta}\},$$ 
then the subgroups $ \Hom_R(M,N)_\gamma$ of $\Hom_R(M,N)$ intersect trivially and $\HOM_R(M,N)$ denotes their direct sum $\bigoplus_{\gamma\in\Gamma} \Hom_R(M,N)_\gamma.$ The notation $\END_R(M)$ is used in the case if $M=N$ and $\END_R(M)$ is a $\Gamma$-graded ring. We use $\AUT_R(M)$ to denote the subring of $\END_R(M)$ consisting of invertible elements. 
If $M$ is finitely generated (which is the case we often consider), then $\Hom_R(M,N)=\HOM_R(M,N)$ for any $N$ (both \cite{NvO_book} and \cite{Roozbeh_book} contain details), $\End_R(M)=\END_R(M)$ is a $\Gamma$-graded ring, and $\Aut_R(M)=\AUT_R(M)$.

In  \cite{Roozbeh_book}, if $R$ is a $\Gamma$-graded ring and $\gamma_1,\dots,\gamma_n\in \Gamma$, then $\M_n(R)(\gamma_1,\dots,\gamma_n)$ denotes the ring of matrices $\M_n(R)$ with the $\Gamma$-grading given by  
\begin{center}
$(r_{ij})\in\M_n(R)(\gamma_1,\dots,\gamma_n)_\delta\;\;$ if $\;\;r_{ij}\in R_{\gamma_i^{-1}\delta\gamma_j}$ for $i,j=1,\ldots, n.$ 
\end{center}
This definition is different in \cite{NvO_book}: $\M_n(R)(\gamma_1,\dots,\gamma_n)$ in \cite{NvO_book} corresponds to $\M_n(R)(\gamma_1^{-1},\dots,\gamma_n^{-1})$ in \cite{Roozbeh_book}. More details on the relations between the two definitions
can be found in \cite[Section 1]{Lia_realization}. Although the definition 
from \cite{NvO_book} has been in circulation longer, some matricial representations of Leavitt path algebras involve positive instead of negative integers making the definition from \cite{Roozbeh_book} more convenient for us. So, we use the definition from \cite{Roozbeh_book}. With this definition, if $F$ is the graded free right module $(\gamma_1^{-1})R\oplus \dots \oplus (\gamma_n^{-1})R,$ then $\Hom_R(F,F)\cong_{\gr} \;\M_n(R)(\gamma_1,\dots,\gamma_n)$ as graded rings.    

We also recall \cite[Remark 2.10.6]{NvO_book} stating the first two parts in Lemma \ref{lemma_on_shifts} and \cite[Theorem 1.3.3]{Roozbeh_book}  stating part (3) for $\Gamma$ abelian. The proof of this statement generalizes to arbitrary $\Gamma.$ 

\begin{lemma}\cite[Remark 2.10.6]{NvO_book}, \cite[Theorem 1.3.3]{Roozbeh_book}.
Let $R$ be a $\Gamma$-graded ring and $\gamma_1,\ldots,\gamma_n\in \Gamma.$ 
\begin{enumerate}[\upshape(1)]
\item If $\pi$ a permutation of the set $\{1,\ldots, n\},$ then 
\begin{center}
$\M_n (R)(\gamma_1, \gamma_2,\ldots, \gamma_n)\;\cong_{\gr}\;\M_n (R)(\gamma_{\pi(1)}, \gamma_{\pi(2)} \ldots, \gamma_{\pi(n)}).$
\end{center}

\item If $\delta$ in the center of $\Gamma,$ $\M_n (R)(\gamma_1, \gamma_2, \ldots, \gamma_n)\;=\;\M_n (R)(\gamma_1\delta, \gamma_2\delta,\ldots, \gamma_n\delta).$

\item If $\delta\in\Gamma$ is such that there is an invertible element $u_\delta$ in $R_\delta,$ then  
\begin{center}
$\M_n (R)(\gamma_1, \gamma_2, \ldots, \gamma_n)\;\cong_{\gr}\;\M_n (R)(\gamma_1\delta, \gamma_2\ldots, \gamma_n).$
\end{center}
\end{enumerate}
\label{lemma_on_shifts} 
\end{lemma}

\subsection{Strongly graded rings, crossed products and skew group rings}\label{subsection_strongly_graded_and_cross}
A graded ring $R$ is {\em strongly graded} if $R_\gamma R_\delta=R_{\gamma\delta}$ for every $\gamma,\delta\in \Gamma.$ If $R$ is unital with the identity $1$, then $R$ is strongly graded if and only if $1\in R_\gamma R_{\gamma^{-1}}$ for any $\gamma\in\Gamma$ (see \cite[Proposition 1.1.1]{NvO_book} or \cite[Proposition 1.1.15]{Roozbeh_book}).

A unital graded ring $R$ is a {\em crossed product} if there is an invertible element in $R_\gamma$ for any $\gamma\in\Gamma$. 
In this case, $R$ is also strongly graded. If $R$ is a graded ring and $M$ and $N$ are graded $R$-modules, $M$ and $N$ are {\em graded weakly isomorphic} if $M\oplus M'\cong_{\gr} N^n$ and $N\oplus N'\cong_{\gr} M^m$ for some positive integers $n,m$ and some graded modules $M'$ and $N'.$ We recall \cite[Theorems 2.10.1 and 2.10.2]{NvO_book}.

\begin{theorem} \cite[Theorems 2.10.1 and 2.10.2]{NvO_book}. Let $R$ be a $\Gamma$-graded ring and $M$ a graded $R$-module. 
\begin{enumerate}[\upshape(1)]
\item The graded ring $\END_R(M)$ is strongly graded if and only if $M$ and $(\gamma)M$ are graded weakly isomorphic for any $\gamma\in \Gamma.$ Thus, if $R$ is unital, $R$ is strongly graded if and only if $R$ and $(\gamma)R$ are graded weakly isomorphic for any $\gamma\in \Gamma.$  
 
\item The graded ring $\END_R(M)$ is a crossed product if and only if $M\cong_{\gr}(\gamma)M$ for any $\gamma\in\Gamma.$ Thus, if $R$ is unital, $R$ is a crossed product if and only if $R\cong_{\gr}(\gamma)R$ for any $\gamma\in\Gamma.$ 
\end{enumerate}
\label{NvO_theorem}
\end{theorem}

If $M$ is a graded module over a graded ring $R$ and $H$ is the set of homogeneous elements in $\AUT_R(M),$ the equivalent conditions from part (2) of the above theorem  are equivalent with
\[\xymatrix{\{1_M\}\ar[r]& \AUT_R(M)_\varepsilon \ar[r]& \AUT_R(M)\cap H\ar[r]&\Gamma\ar[r]& \{\varepsilon\}}\]
being exact (in the category of groups) where the first two maps are inclusions and the third map is the degree map (\cite[Proposition 1.4.1]{NvO_book}). 

If $R$ is a unital ring which is a crossed product, then $R$ is a {\em skew group ring} if the above sequence splits for $M=R$ (see \cite[Remark 1.4.3]{NvO_book}). In this case, the splitting uniquely determines a map $\sigma$ from $\Gamma$ to the set of invertible homogeneous elements of $R.$ Such $\sigma$ defines $\phi:\Gamma\to \Aut_{R_\varepsilon}(R_\varepsilon)$ by mapping $\gamma$ to the conjugation by $\sigma(\gamma).$ If $R_\varepsilon \ast_\phi \Gamma$ denotes a free $R_\varepsilon$-module with the basis $\Gamma,$ then the multiplication $(r\gamma)(s\delta)=(r\phi(\gamma)(s))\gamma\delta$ for $r,s\in R_\varepsilon$ and $\gamma,\delta\in \Gamma$ makes $R_\varepsilon \ast_\phi \Gamma$ into a $\Gamma$-graded ring such that $R\cong_{\gr} R_\varepsilon \ast_\phi \Gamma.$ 

A skew group ring $R$ is a {\em group ring} if  $\{1\}$ is the image of $\sigma$. In this case we use $R_\varepsilon[\Gamma]$ for $R_\varepsilon \ast_\phi \Gamma.$ 

Recall that a graded unital ring $R$ is {\em graded unit-regular} if for every $a\in R_\gamma,$ there is an invertible element $u$ in $R_{\gamma^{-1}}$ such that $a=aua.$ If this holds, then every nonzero component $R_\gamma$ of $R$ has an invertible element (this is direct to check or see \cite[Lemma 2.5]{Lia_cancellation}). Hence, the following holds. 

\begin{proposition} A strongly graded, unital ring which is graded unit-regular is a crossed product. 
\label{ur_and_sg_implies_cross_product}
\end{proposition}
\begin{proof}
Let $R$ be a strongly graded, unital and graded unit-regular ring. Since $R$ is graded unit-regular, if $R_\gamma$ is nonzero for some $\gamma\in\Gamma$, then $R_\gamma$ has an invertible element. On the other hand, since $R$ is strongly graded, the ring identity is in $R_\gamma R_{\gamma^{-1}}$ for every $\gamma\in\Gamma,$ so $R_\gamma$ is nonzero for every $\gamma\in \Gamma.$ Thus, $R_\gamma$ has an invertible element for every $\gamma\in \Gamma.$ So, $R$ is a crossed product. 
\end{proof}

The converse of this proposition does not have to hold as the following example shows. Let $\Gamma$ be the infinite cyclic group on a generator $x$ and $R$ be the ring of Laurent polynomials $\Zset[\Gamma]=\Zset[x,x^{-1}]$ graded by $\Gamma$ so that $R_{x^n}=\{kx^n |k\in \Zset\}.$ Then $R$ is a group ring so it is a crossed product. However, $R$ is not graded unit regular since $2\in R_\varepsilon$ is such that $2\neq 2u2$ for any invertible element $u\in R_\varepsilon.$ 

While the converse of Proposition \ref{ur_and_sg_implies_cross_product} does not hold in general, Corollary \ref{unit_reg_corollary} shows that the converse holds for unital Leavitt path algebras.

\subsection{Pre-ordered \texorpdfstring{$\Gamma$}{TEXT}-groups}\label{subsection_ordered_groups}

If $\Gamma$ is a group and $M$ an additive monoid with a left action of $\Gamma$ such that $\gamma(g_1+g_2)=\gamma g_1+\gamma g_2$ for all $\gamma\in \Gamma$ and $g_1,g_2\in M,$ we say that $M$ is a {\em $\Gamma$-monoid}. If $\Gamma$ is a group and $G$ an abelian group with a left action of $\Gamma$ which satisfies this property, $G$ is a {\em $\Gamma$-group}. 

Let $\geq$ be a reflexive and transitive relation (a pre-order) on a $\Gamma$-monoid $M$ ($\Gamma$-group $G$) such that $g_1\geq g_2$ implies $g_1 + h\geq g_2 + h$ and $\gamma g_1 \geq \gamma g_2$ for all $g_1, g_2, h$ in $M$ (in $G$) and $\gamma\in \Gamma.$ We say that such monoid $M$ is a {\em pre-ordered $\Gamma$-monoid} and that such a group $G$ is a {\em pre-ordered $\Gamma$-group}. 

If $G$ is a pre-ordered $\Gamma$-group, the set $G^+=\{ x\in G\,|\, x\geq 0\},$ called the {\em positive cone} of $G,$ is a $\Gamma$-monoid. Conversely, any additively closed subset $M$ of $G$ which contains 0 and is closed under the action of $\Gamma$ is a $\Gamma$-monoid and it defines a pre-order $\Gamma$-group structure on $G$ with $G^+=M$. The positive cone $G^+$ is {\em strict} if $G^+\cap (-G^+)=\{0\}.$ This is equivalent with $\leq$ being an order in which case $G$ is an {\em ordered $\Gamma$-group}. If $\Zset^+$ is the set of nonnegative integers, $\Delta$ a subgroup of $\Gamma,$ and $\Gamma/\Delta$ the set of left cosets, $\Zset[\Gamma/\Delta]$ is an ordered $\Gamma$-group with the order given by $\sum k_\gamma\gamma\Delta\geq \sum m_\gamma\gamma\Delta$ if $k_\gamma\geq m_\gamma$ for all $\gamma.$ The monoid $\Zset^+[\Gamma/\Delta]$ is the strict positive cone of this order.  
 
An element $u$ of a pre-ordered $\Gamma$-monoid $M$ is an \emph{order-unit} if for any $x\in M,$ $x\leq au$ for some $0\neq a\in \Zset^+[\Gamma].$ An element $u$ of a pre-ordered $\Gamma$-group $G$ is an \emph{order-unit} if $u\in G^+$ and for any $x\in G$, $x\leq au$ for some $0\neq a\in \Zset^+[\Gamma].$ In this case, $-x\leq bu$ for some $0\neq b\in \Zset^+[\Gamma]$ so that $-bu\leq x$ and thus $-(a+b)u\leq -bu\leq x\leq au\leq (a+b)u.$ It is direct to check that an order-unit of $G^+$ is an order-unit of $G.$ 

If $G$ and $H$ are pre-ordered $\Gamma$-groups, $f: G\to H$ is a homomorphism of pre-ordered $\Gamma$-groups if $f$ is a group homomorphism which is order preserving ($f(G^+)\subseteq H^+$) and  $\Gamma$-equivariant. If $G$ and $H$ also have order-units $u$ and $v$ respectively, and such $f$ preserves them (i.e.  $f(u)=v$), we write $f:(G, u)\to (H, v).$ If such $f$ is a bijection, we write $f:(G, u)\cong (H, v).$

Consider the ordered $\Gamma$-group $\Zset[\Gamma/\Delta]$ from the example above for some subgroup $\Delta$ of $\Gamma$. One directly checks that $\Delta$ is an order-unit of $\Zset[\Gamma/\Delta].$ For $a\in \Zset[\Gamma/\Delta],$ let 
$$\Stab(a)=\{\gamma\in\Gamma\, |\, \gamma a=a \}\mbox{ and }O(a)=\{\gamma a\in \Zset[\Gamma/\Delta]\, |\,\gamma\in\Gamma\}$$
denote the stabilizer subgroup and the orbit of $a.$ For $\gamma\in\Gamma,$ $\Stab(\gamma\Delta)=\gamma\Delta\gamma^{-1}$  and $O(\gamma\Delta)=\Gamma/\Delta.$

\subsection{The Grothendieck \texorpdfstring{$\Gamma$}{TEXT}-group}\label{subsection_Grothendieck_graded}

If $R$ is a $\Gamma$-graded unital ring, let $\V^{\Gamma}(R)$ denote the monoid of graded isomorphism classes $[P]$ of finitely generated graded projective right $R$-modules $P$ with the direct sum as the addition operation, the usual pre-order, and the left $\Gamma$-action given by 
\[(\gamma, [P])\mapsto [(\gamma^{-1})P]\] making $\V^{\Gamma}(R)$ a pre-ordered $\Gamma$-monoid. The $\Gamma$-monoid structure can also be defined using the left modules instead of the right, in which case the $\Gamma$ action is given by $(\gamma, [P])\mapsto [P(\gamma)],$ and the two approaches are equivalent by \cite[Section 2.4]{NvO_book} or \cite[Section 1.2.3]{Roozbeh_book}. The $\Gamma$-monoid $\V^{\Gamma}(R)$ can also be represented using the conjugation classes of homogeneous idempotent matrices as has been done in \cite[Section 3.2]{Roozbeh_book} for abelian groups $\Gamma$ and in \cite[Section 1.3]{Lia_realization} for arbitrary $\Gamma.$  

The \emph{Grothendieck $\Gamma$-group} $K_0^{\Gamma}(R)$ is the group completion of the $\Gamma$-monoid $\V^{\Gamma}(R)$ which naturally inherits the action of $\Gamma$ from $\V^{\Gamma}(R)$. The image of $\V^{\Gamma}(R)$ under the natural map $\V^{\Gamma}(R)\to K_0^{\Gamma}(R)$ is a positive cone making $K_0^{\Gamma}(R)$ into a pre-ordered $\Gamma$-group. The element $[R]$ is an order-unit (see \cite[Example 3.6.3]{Roozbeh_book}). If $\Gamma$ is trivial, $K_0^{\Gamma}(R)$ is the usual $K_0$-group. In \cite{Roozbeh_book}, the author uses ``the graded Grothendieck group'' instead of ``the Grothendieck $\Gamma$-group'' and $K_0^{\gr}(R)$ instead of $K_0^{\Gamma}(R).$ In this paper, we use  ``the Grothendieck $\Gamma$-group'' since the group $K_0^{\Gamma}(R)$  is not itself graded by $\Gamma.$

The proof of \cite[Proposition 1.3.16]{Roozbeh_book} holds even if $\Gamma$ is not necessarily abelian as was pointed out in \cite[Section 1.4]{Lia_realization}. Thus, $(\gamma)R$ is graded isomorphic to $(\delta)R$ if and only if there is invertible element $a\in R_{\gamma\delta^{-1}}$ ($a^{-1}$ is then necessarily in $R_{\delta\gamma^{-1}}).$ Hence, if $K$ is a $\Gamma$-graded division ring,
\begin{center}
$(\gamma)K\cong_{\gr}(\delta)K$ iff $\gamma\delta^{-1}\in \Gamma_K$  iff $\gamma^{-1}\Gamma_K=\delta^{-1}\Gamma_K$ 
\end{center}
for any $\gamma,\delta\in\Gamma$ by \cite[Corollary 1.3.17]{Roozbeh_book}. So, $(K_0^{\Gamma}(K), [K])\cong (\Zset[\Gamma/\Gamma_K], \Gamma_K)$ by  the map given by 
$$[\bigoplus_{i=1}^n (\gamma_i^{-1})K^{n_i}]=\sum_{i=1}^n n_i\gamma_i[K] \mapsto \sum_{i=1}^n n_i\, \gamma_i\Gamma_K.$$  

If $R=\M_n(K)(\gamma_1,\ldots, \gamma_n)$ for some $\gamma_1,\ldots, \gamma_n\in \Gamma$ and $e_{ij}$ is the $(i,j)$-th standard graded matrix unit, then $[R]=[\bigoplus_{i=1}^n e_{ii}R]=\sum_{i=1}^n \gamma_i\gamma_1^{-1}[e_{11}R].$ Also, $[e_{ii}R]\in \V^\Gamma(R)$ corresponds to $[e_{ii}(\bigoplus_{i=1}^n(\gamma_i^{-1})K)]=[(\gamma_i^{-1})K]\in \V^\Gamma(K)$ under the isomorphism induced by the equivalence from \cite[Corollary 2.1.2]{Roozbeh_book} (also \cite[Section 1.4]{Lia_realization}). Thus, $(K_0^{\Gamma}(R)), [R])\cong (\Zset[\Gamma/\Gamma_K], \sum_{i=1}^n \gamma_i\Gamma_K).$

\subsection{Strong order-unit}\label{subsection_strong_unit}
If $M$ is a pre-ordered $\Gamma$-monoid and $u\in M,$ we say that $u$ is a {\em strong order-unit} of $M$ if for any $x\in M$, there is $0\neq n\in\Zset^+$ such that $x\leq nu$ or, equivalently, if $u$ is an order-unit of $M$ if $M$ is considered as a pre-ordered $\{\varepsilon\}$-monoid. An element $u$ of a pre-ordered $\Gamma$-group $G$ is a  {\em strong order-unit} of $G$ if 
$u\in G^+$ and for any $x\in G$, there is $0\neq n\in\Zset^+$ such that $x\leq nu.$ The following implications justify our terminology.
\begin{center}
\begin{tabular}{lllclll}
$R$ is a && graded ring & $\Longrightarrow$ & $[R]$ is an&& order-unit of $\V^\Gamma(R)$ (of $K_0^\Gamma(R)$) \\
$R$ is a &{\em strongly}& graded ring & $\Longrightarrow$ & $[R]$ is a &{\em strong}& order-unit of $\V^\Gamma(R)$ (of $K_0^\Gamma(R)$)\\
\end{tabular}
\end{center} 
The next proposition shows the second implication above as well as its converse.  
\begin{proposition}
If $R$ is a $\Gamma$-graded unital ring, then  $R$ is strongly graded if and only if  $[R]$ is a strong order-unit of $\V^\Gamma(R)$ (equivalently, of $K_0^\Gamma(R)$). 
\label{strong_order_unit_prop}
\end{proposition}
\begin{proof}
If $R$ is strongly graded and $[P]\in \V^\Gamma(R),$ then $[P_\varepsilon]\in \V(R_\varepsilon)$ so $[P_\varepsilon]\leq n [R_\varepsilon]$ for some $0\neq n\in\Zset^+.$ As $[P]=[P_\varepsilon\otimes_{R_\varepsilon}R]$ and $[R]=[R_\varepsilon\otimes_{R_\varepsilon}R]$ (see \cite[Theorem 1.5.1]{Roozbeh_book}) the relation $[P_\varepsilon]\leq n [R_\varepsilon]$ implies $[P]\leq n [R].$ 

By Theorem \ref{NvO_theorem}, to prove the converse, it is sufficient to show that for every $\delta\in\Gamma,$ $[R] \leq k(\delta)[R]$ and $(\delta)[R]\leq l[R]$ hold in $\V^\Gamma(R)$ for some positive integers $k$ and $l$. If $[R]$ is a strong order-unit of $\V^\Gamma(R)$ and $\delta\in \Gamma$ is arbitrary, then $(\delta)[R]\leq l[R]$ holds for some positive integer $l$. Also,  $(\delta^{-1})[R]\leq k[R]$ for some positive integer $k$ and so  $[R]\leq k(\delta)[R].$
\end{proof}

\subsection{Leavitt path algebras}\label{subsection_LPAs}
Let $E$ be a directed graph and let $E^0$ denote the set of vertices, $E^1$ the set of edges, and $\so$ and $\ra$ the source and the range maps. The graph $E$ is {\em row-finite} if $\so^{-1}(v)$ is finite for every $v\in E^0$ and $E$ is {\em finite} if both $E^0$ and $E^1$ are finite. A vertex $v$ is a {\em sink} if $\so^{-1}(v)$ is empty and $v$ is {\em regular} if $\so^{-1}(v)$ is finite and nonempty. A {\em cycle} is a closed path such that different edges in the path have different sources. A cycle {\em has an exit} if a vertex on the cycle emits an edge outside of the cycle. The graph $E$ is {\em acyclic} if there are no cycles and {\em no-exit} if no cycle has an exit.  

An {\em infinite path} of a graph $E$ is a sequence of edges $e_1e_2\ldots$ such that $\ra(e_i)=\so(e_{i+1})$ for $i=1,2,\ldots$. An infinite path is an \emph{infinite sink} if none of its vertices emits more than one edge or is on a cycle. An infinite path \emph{ends in a sink} if there is a positive integer $n$ such that the subpath $e_ne_{n+1}\hdots$ is an infinite sink. An infinite path \emph{ends in a cycle} if there is a positive integer $n$ such that the subpath $e_ne_{n+1}\hdots$ is equal to the path $cc\hdots$ for some cycle $c.$ 

If $K$ is any field, the \emph{Leavitt path algebra} $L_K(E)$ of $E$ over $K$ is a free $K$-algebra generated by the set  $E^0\cup E^1\cup\{e^*\ |\ e\in E^1\}$ such that for all vertices $v,w$ and edges $e,f,$

\begin{tabular}{ll}
(V)  $vw =0$ if $v\neq w$ and $vv=v,$ & (E1)  $\so(e)e=e\ra(e)=e,$\\
(E2) $\ra(e)e^*=e^*\so(e)=e^*,$ & (CK1) $e^*f=0$ if $e\neq f$ and $e^*e=\ra(e),$\\
(CK2) $v=\sum_{e\in \so^{-1}(v)} ee^*$ for each regular vertex $v.$ &\\
\end{tabular}

By the first four axioms, every element of $L_K(E)$ can be represented as a sum of the form $\sum_{i=1}^n a_ip_iq_i^*$ for some $n$, paths $p_i$ and $q_i$, and elements $a_i\in K,$ for $i=1,\ldots,n.$ Using this representation, it is direct to see that $L_K(E)$ is a unital ring if and only if $E^0$ is finite in which case the sum of all vertices is the identity. For more details on this and other basic properties, see \cite{LPA_book}.

A Leavitt path algebra is naturally graded by the group of integers $\Zset$ so that the $n$-component $L_K(E)_n$ is  the $K$-linear span of the elements $pq^*$ for paths $p, q$ with $|p|-|q|=n$ where $|p|$ denotes the length of a path $p.$ One can grade a Leavitt path algebra by any group $\Gamma$ when considering a function $w:E^1\to \Gamma,$ called the {\em weight} function (see \cite[Section 1.6.1]{Roozbeh_book}).  

If $K$ is a trivially $\Zset$-graded field, let $K[x^m, x^{-m}]$ be the graded field of Laurent polynomials $\Zset$-graded by $K[x^m, x^{-m}]_{mk}=Kx^{mk}$ and $K[x^m, x^{-m}]_{n}=0$ if $m$ does not divide $n.$  
By \cite[Proposition 5.1]{Roozbeh_Lia_Ultramatricial}, if $E$ is a finite no-exit graph, then $L_K(E)$ is graded isomorphic to 
$$R=\bigoplus_{i=1}^k \M_{k_i} (K)(\gamma_{i1}\ldots,\gamma_{ik_i}) \oplus \bigoplus_{j=1}^n \M_{n_j} (K[x^{m_j},x^{-m_j}])(\delta_{j1}, \ldots, \delta_{jn_j})$$
where $k$ is the number of sinks, $k_i$ is the number of paths ending in the $i$-th sink for $i=1,\ldots, k,$ and $\gamma_{il}$ is the length of the $l$-th path ending in the $i$-th sink for $l=1,\ldots, k_i$ and $i=1,\ldots, k.$ In the second term, $n$ is the number of cycles, $m_j$ is the length of the $j$-th cycle for $j=1,\ldots, n,$ $n_j$ is the number of paths which do not contain the cycle indexed by $j$ and which end in a fixed but arbitrarily chosen vertex of the cycle, and $\delta_{jl}$ is the length of the $l$-th path ending in the fixed vertex of the $j$-th cycle for $l=1,\ldots, n_j$ and $j=1,\ldots, n.$ This representation is not necessarily unique as  \cite[Example 2.2]{Lia_LPA_realization} shows, but it is unique up to a graded isomorphism by Lemma \ref{lemma_on_shifts}. The algebra $R$ is as a {\em graded matricial representation} of $L_K(E).$ 

\subsection{The Grothendieck \texorpdfstring{$\Gamma$}{TEXT}-group of a graph}\label{subsection_graph_group}

If $E$ is a graph and $\Gamma$ is a group, the authors of \cite{Ara_et_al_Steinberg} construct a commutative monoid $M^\Gamma_E$ which is isomorphic to $\V^{\Gamma}(L_K(E))$ as a $\Gamma$-monoid. The authors of \cite{Roozbeh_Lia_comparability} provide an alternative construction of $M^\Gamma_E$ which we briefly review and recall one result of \cite{Roozbeh_Lia_comparability}. 

Let $F_E^\Gamma$ be a free commutative $\Gamma$-monoid generated by $v\in E^0$ and elements $q_Z$ for any finite and nonempty subset $Z$ of $\so^{-1}(v)$ if $v$ is an infinite emitter. Thus, any $a\in F_E^\Gamma-\{0\}$ can be written as $\sum_{i=1}^n \gamma_ig_i$ for some positive integer $n,$ $\gamma_i\in \Gamma,$ and generators $g_i$ of $F_E^\Gamma,$ possibly repeated, for  $i=1,\ldots,n.$  This representation, unique up to a permutation, is a {\em normal representation} of $a.$ The set $\{g_i |i=1,\ldots, n\}$ is the {\em support} of $a.$ 

If $w:E^1\to \Gamma$ is any weight function, the monoid $M_E^\Gamma$ is the quotient of $F_E^\Gamma$ under the congruence closure $\sim$ of the relation $\to_1$ defined by the three conditions below
for any  $\gamma\in \Gamma$ and $a\in F_E^\Gamma.$  
\begin{enumerate}
\item[(A1)] If $v$ is a regular vertex, then  
\[a+\gamma v\to_1 a+\sum_{e\in s^{-1}(v)}\gamma w(e)\ra(e).\]
\item[(A2)] If $v$ is an infinite emitter and $Z$ a finite and nonempty subset of $\so^{-1}(v),$ then  
\[a+\gamma v\to_1 a+\gamma q_Z+\sum_{e\in Z}\gamma w(e)\ra(e).\]
\item[(A3)] If $v$ is an infinite emitter and $Z\subsetneq W$ are finite and nonempty subsets of $\so^{-1}(v),$ then 
\[a+\gamma q_Z\to_1 a+\gamma q_W+\sum_{e\in W-Z}\gamma w(e)\ra(e).\]
\end{enumerate}
One often considers an intermediate step in the construction of $\sim$ from $\to_1$ and let $\to$ be the reflexive and transitive closure of $\to_1.$ 
Then, $\sim$ can be defined as the congruence closure of $\to$. The following lemma was first shown for trivial $\Gamma$ in \cite[Proposition 5.9]{Ara_Goodearl} in order to show that the monoid $M_E^{\{\varepsilon\}}$ has the refinement property. This proof was adapted to arbitrary $\Gamma$ in \cite[Lemma 5.9]{Ara_et_al_Steinberg}. 

\begin{lemma} \cite[Lemma 5.9]{Ara_et_al_Steinberg}.  
Let $E$ be a graph, $\Gamma$ a group, $w:E^1\to\Gamma$ a weight map and $a,b\in  F^\Gamma_E-\{0\}.$ The relation $a\sim b$ holds if and only if $a \to c$ and $b\to c$ for some $c \in F_E-\{0\}.$  
\label{confluence}
\end{lemma}

If $[a]$ denotes the congruence class of $a\in F^\Gamma_E,$ $\leq$, defined below, is a  pre-order on $M_E^\Gamma.$ 
\[[a]\leq [b]\mbox{ if there is $c\in F^\Gamma_E$ such that }a+c\sim b\]

By \cite[Corollary 5.8]{Ara_et_al_Steinberg} (also by \cite[Proposition 3.1]{Roozbeh_Lia_comparability}), if $\Gamma$ is the infinite cyclic group on a generator $x,$ the monoid $M_E^\Gamma$ is cancellative. So, $\leq$ is an order and $M_E^\Gamma$ is the positive cone of its Grothendieck group $G^\Gamma_E.$ Also, if $E^0$ is finite and $1_E=\sum_{v\in E^0}[v],$ then $1_E$ is an order-unit of $M_E^\Gamma$ (and $G_E^\Gamma$).  

If  $a\in F_E^\Gamma$ and  $[a]=x^n[a]$ for some positive integer $n,$ the element $[a]$ is said to be {\em periodic}.
In \cite[Theorem 4.1]{Roozbeh_Lia_comparability}, the authors characterize this property of $[a]$ in terms of the properties of the generators in a normal representation of $a\in F_E^\Gamma.$ We recall this result below. 

\begin{theorem} \cite[Theorem 4.1]{Roozbeh_Lia_comparability}. If $\Gamma$ is the infinite cyclic group on a generator $x,$ the following conditions are equivalent for an element $a\in F^\Gamma_E-\{0\}.$
\begin{enumerate}[\upshape(1)]
\item The element $[a]\in M_E^\Gamma-\{0\}$ is periodic. 
\item Any path originating at a generator in the support of $a$ is a prefix of a path $p$ ending in one of finitely many cycles with no exits and such that all vertices of $p$ are regular. Every infinite path originating at a vertex in the support of $a$ ends in a cycle with no exits.
\end{enumerate}
\label{periodic}
\end{theorem}

\section{Strongly graded, crossed product and group ring characterizations of graded matrix algebras over graded division rings}\label{section_cross_product}

In this section, we characterize when a graded matrix ring $R$ over a graded division ring $K$ is strongly graded, a crossed product, a skew group ring or a group ring in terms of the support $\Gamma_K$ of $K$ and the $\Gamma$-action on $K_0^\Gamma(R)$ for any $\Gamma$.  
Recall that if $R=\M_n(K)(\gamma_1, \dots,  \gamma_n),$ then $$(\V^\Gamma(R), [R])\cong (\Zset^+[\Gamma/\Gamma_K], \sum_{i=1}^n \gamma_i\Gamma_K)\;\;\;\;\mbox{ and }\;\;(K_0^\Gamma(R), [R])\cong (\Zset[\Gamma/\Gamma_K], \sum_{i=1}^n \gamma_i\Gamma_K).$$ 

\begin{theorem} Let $K$ be a $\Gamma$-graded division ring with support $\Gamma_K,$ $\gamma_1, \dots,  \gamma_n\in\Gamma,$ let $R$ denote the $\Gamma$-graded matrix ring $\M_n(K)(\gamma_1, \dots,  \gamma_n),$ and $u=\sum_{i=1}^n \gamma_i\Gamma_K$
denote an order-unit of $\Zset[\Gamma/\Gamma_K]$. 
\begin{enumerate}[\upshape(i)]
\item The following conditions are equivalent. 

\begin{enumerate}[\upshape(1)]
\item $R$ is strongly graded. 
\smallskip 
\item $\Gamma/\Gamma_K=\{\gamma_1\Gamma_K, \dots, \gamma_n\Gamma_K\}.$ 
\smallskip
\item The orbit of $\gamma_i\Gamma_K$ in $\Zset[\Gamma/\Gamma_K]$ is $\{\gamma_1\Gamma_K, \dots, \gamma_n\Gamma_K\}$ for every $i=1,\ldots, n.$
\item The order-unit $u$ is a strong order-unit of $\Zset[\Gamma/\Gamma_K]$.  
\end{enumerate} 
\medskip

\item Let $\gamma_{i_1},\ldots,\gamma_{i_m}$ be a complete list of coset representatives of $\gamma_1\Gamma_K, \dots, \gamma_n\Gamma_K$ and let $k_j$ denote the number of times $\gamma_{i_j}\Gamma_K$ appears in the list $\gamma_1\Gamma_K, \dots, \gamma_n\Gamma_K$ so that $u=\sum_{j=1}^m k_j\gamma_{i_j}\Gamma_K.$ The following conditions are equivalent. 

\begin{enumerate}[\upshape(1)]
\item $R$ is a crossed product.\smallskip
 
\item $\Gamma/\Gamma_K=\{\gamma_1\Gamma_K, \dots, \gamma_n\Gamma_K\}$ and $k_1=k_2=\ldots=k_m.$
\smallskip
\item $R$ is a skew group ring. 
\smallskip
\item $O(\gamma_i\Gamma_K)=\{\gamma_1\Gamma_K, \dots, \gamma_n\Gamma_K\}$ for every $i=1,\ldots, n\;$ and $\;\Stab(u)=\Gamma$ where the orbit and the stabilizer are taken in $\Zset[\Gamma/\Gamma_K]$.
\end{enumerate}
\medskip

\item The following conditions are equivalent. 

\begin{enumerate}[\upshape(1)]
\item $R$ is group ring.
\smallskip

\item $\Gamma=\Gamma_K.$
\smallskip

\item For every $a\in \Zset[\Gamma/\Gamma_K],$ $\Stab(a)=\Gamma.$

\item $\Gamma$ acts trivially on $\Zset[\Gamma/\Gamma_K]$. 
\end{enumerate}
\end{enumerate}
The above statements hold if $\Zset[\Gamma/\Gamma_K]$ is replaced by  $\Zset^+[\Gamma/\Gamma_K]$ in {\em (i3), (i4), (ii4), (iii3),} and {\em (iii4)}.
\label{three_characterizations} 
\end{theorem}
\begin{proof}
(i) We show (1) $\Rightarrow$ (2) first. Let $F=\bigoplus_{i=1}^n (\gamma_i^{-1})K$ so that $R\cong_{\gr}\End_K(F).$ 
By Theorem \ref{NvO_theorem}, for any $\delta\in\Gamma,$ there is a finitely generated graded projective module $P',$ positive integer $k,$ and a  graded isomorphism $F\oplus P'\cong_{\gr} (\delta)F^k.$ This isomorphism maps $F_\lambda\oplus P'_\lambda$ onto $F_{\delta\lambda}^k$ for any $\lambda\in\Gamma.$ Thus, for $\lambda=\gamma_1,$ we have that $0\neq K_\varepsilon\leq F_{\gamma_1}\oplus P'_{\gamma_1}$ is mapped to $F_{\delta\gamma_1}^k$ which implies that $0\neq F_{\delta\gamma_1}=\bigoplus_{i=1}^n K_{\gamma_i^{-1}\delta\gamma_1}.$ So, for every $\delta,$ there is $i=1,\ldots, n$ such that the equivalent conditions below hold. 
\begin{center}
$K_{\gamma_i^{-1}\delta\gamma_1}\neq 0\;\;$ if and only  $\;\;\gamma_i^{-1}\delta\gamma_1\in\Gamma_K\;\;$ if and only if $\;\;\delta\gamma_1\Gamma_K=\gamma_i\Gamma_K.$ 
\end{center}
This shows that for any $\delta\in \Gamma,$ there is $i=1,\ldots, n$ such that $\delta\Gamma_K=\delta\gamma_1^{-1}\gamma_1\Gamma_K=\gamma_i\Gamma_K$ and so $\Gamma/\Gamma_K=\{\gamma_1\Gamma_K, \dots, \gamma_n\Gamma_K\}.$ 

Let us show (2) $\Rightarrow$ (1) now. If $\Gamma/\Gamma_K=\{\gamma_1\Gamma_K, \dots, \gamma_n\Gamma_K\},$ then $\{\delta^{-1}\gamma_1\Gamma_K, \ldots, \delta^{-1}\gamma_n\Gamma_K\}$ is also equal to $\Gamma/\Gamma_K$ for any $\delta\in\Gamma.$ Let $\gamma_{i_1},\ldots,\gamma_{i_m}$  be a complete list of coset representatives of $\gamma_1\Gamma_K, \dots, \gamma_n\Gamma_K$ and let $k_j$ denote the number of times $\gamma_{i_j}\Gamma_K$ appears in the list  $\gamma_1\Gamma_K, \dots, \gamma_n\Gamma_K.$ The terms of $F=\bigoplus_{i=1}^n (\gamma_i^{-1})K$ can be permuted so that $F\cong_{\gr}\bigoplus_{j=1}^m (\gamma_{i_j}^{-1})K^{k_j}.$

Let $l_j$ be the multiplicity of $\gamma_{i_j}\Gamma_K$ in the list $\delta^{-1}\gamma_1\Gamma_K, \ldots, \delta^{-1}\gamma_n\Gamma_K$ so that we have that \[(\delta)F=\bigoplus_{i=1}^n (\gamma_i^{-1}\delta)K\cong_{\gr}\bigoplus_{j=1}^m (\gamma_{i_j}^{-1}\delta)K^{k_j}\cong_{\gr}\bigoplus_{j=1}^m (\gamma_{i_j}^{-1})K^{l_j}.\]
Let $k$ be an integer such that $kl_j-k_j\geq 0$ for each $j=1,\ldots, m$ and let $P'=\bigoplus_{j=1}^m (\gamma_{i_j}^{-1})K^{kl_j-k_j}.$ Then, \[F\oplus P'\cong_{\gr}\bigoplus_{j=1}^m (\gamma_{i_j}^{-1})K^{k_j}\oplus\bigoplus_{j=1}^m (\gamma_{i_j}^{-1})K^{kl_j-k_j}\cong_{\gr}  \bigoplus_{j=1}^m (\gamma_{i_j}^{-1})K^{kl_j}\cong_{\gr} (\delta)F^k.\]
Similarly, one can show that $(\delta)F\oplus Q'\cong_{\gr}F^l$ for some positive integer $l$ and  finitely generated graded projective module $Q'.$ Thus, $R$ is strongly graded by Theorem \ref{NvO_theorem}. 

The equivalence (2) $\Leftrightarrow$ (3) is direct by the formula $O(\gamma\Gamma_K)=\Gamma/\Gamma_K$ for any $\gamma\in\Gamma$. 

The equivalence (1) $\Leftrightarrow$ (4) follows from Proposition \ref{strong_order_unit_prop}. 

(ii) Assuming that (1) holds, consider $\bigoplus_{j=1}^m (\gamma_{i_j}^{-1})K^{k_j}$ as in the proof of (i2) $\Rightarrow$ (i1). By Theorem \ref{NvO_theorem}, $R$ is a crossed product if and only if 
\[ \bigoplus_{j=1}^m (\gamma_{i_j}^{-1})K^{k_j}\cong_{\gr}\bigoplus_{j=1}^m (\gamma_{i_j}^{-1}\delta)K^{k_j}\]
for any $\delta\in\Gamma.$ Consider the case when $\delta$ is $\gamma_{i_j}$ for any $j,$ and then consider the $\gamma_{i_l}$-component of the two modules in the formula above for $l=1,\ldots, m$ such that $\gamma_{i_l}\Gamma_K=\Gamma_K.$ 
Since 
\[ \bigoplus_{j'=1}^m (\gamma_{i_{j'}}^{-1})K^{k_{j'}}_{\gamma_{i_l}}=\bigoplus_{j'=1}^m K^{k_{j'}}_{\gamma_{i_{j'}}^{-1}\gamma_{i_l}}=K^{k_l}_{\gamma_{i_l}^{-1}\gamma_{i_l}}=K_\varepsilon^{k_l} \;\mbox{ and }\;\bigoplus_{j'=1}^m(\gamma_{i_{j'}}^{-1}\gamma_{i_j})K^{k_{j'}}_{\gamma_{i_l}}=\bigoplus_{j'=1}^m K^{k_{j'}}_{\gamma_{i_{j'}}^{-1}\gamma_{i_j}\gamma_{i_l}}=K^{k_j}_{\gamma_{i_l}}\cong K^{k_j}_\varepsilon,\]
we have that $K_\varepsilon^{k_l}\cong K^{k_j}_\varepsilon.$ Thus, $k_j=k_l.$ Since $j$ was arbitrary, the condition (2) holds. 

Assuming that (2) holds, $F=\bigoplus_{i=1}^n (\gamma_i^{-1})K$ and $(\gamma)F$ differ only in the order of the terms, so they are graded isomorphic for every $\gamma\in\Gamma$. Hence, $R$ is a crossed product and the sequence 
\[\xymatrix{\{1_F\}\ar[r]& \Aut_K(F)_\varepsilon \ar[r]& \Aut_K(F)\cap H\ar[r]^{\hskip.9cm\psi} &\Gamma\ar[r]& \{\varepsilon\}}\]
is exact where $H$ is the set of homogeneous elements of $\Aut_K(F)$ and $\psi$ is the degree map. To show (3), we define a splitting $\phi:\Gamma\to \Aut_K(F)\cap H.$  

For $\gamma\in\Gamma,$ let $\phi_\gamma\in\Aut_K(F)_\gamma$ be an isomorphism $F\cong_{\gr}(\gamma)F$ and let $\pi$ be the permutation of $\{1,2,\ldots, n\}$ which corresponds to the reordering of the terms done by $\phi_\gamma$ so that  
\begin{center} 
$\phi_\gamma((\gamma_i^{-1})K)=(\gamma_{\pi(i)}^{-1})K\;\;$ and $\;\;\gamma_i\Gamma_K=\gamma^{-1}\gamma_{\pi(i)}\Gamma_K.$
\end{center}

If $\delta\in \Gamma$ is such that $\phi_\delta((\gamma_i^{-1})K)=(\gamma_{\sigma(i)}^{-1})K$ for a permutation $\sigma,$ then 
$\gamma_i\Gamma_K=\delta^{-1}\gamma_{\sigma(i)}\Gamma_K=\delta^{-1}\gamma^{-1}\gamma_{\pi\sigma(i)}\Gamma_K=(\gamma\delta)^{-1}\gamma_{\pi\sigma(i)}\Gamma_K.$ So, the composition $\pi\sigma$ corresponds to the permutation of the terms of $F$ by the graded isomorphism $\phi_{\gamma\delta}.$ Thus, 
$\phi_\gamma\phi_\delta((\gamma_i^{-1})K)=\phi_\gamma((\gamma_{\sigma(i)}^{-1})K)=(\gamma_{\pi\sigma(i)}^{-1})K$ which shows that $\phi_\gamma\phi_\delta=\phi_{\gamma\delta}$ and that the mapping $\gamma\mapsto \phi_\gamma$ defines a group homomorphism $\phi:\Gamma\to\Aut_K(F)\cap H.$
Since $\psi\phi$ is the identity, the short exact sequence above splits. Hence, $R$ is a skew group ring and (3) holds. 

The implication (3) $\Rightarrow$ (1) is direct and we show that (2) $\Leftrightarrow$ (4) next. If (2) holds, then $\gamma u=\gamma \sum_{j=1}^m k_j\gamma_{i_j}\Gamma_K=\sum_{j=1}^m k_j\gamma_{i_j}\Gamma_K=u$ and so $\Stab(u)=\Gamma$ and (4) holds. To show the converse, assume that (4) holds and let $j,j'\in\{1,\ldots, m\}$ be arbitrary. Let $l=1,\ldots, m$ be such that $\gamma_{i_l}\gamma_{i_j}\Gamma_K=\gamma_{i_{j'}}\Gamma_K.$ Then $\gamma_{i_l}u=u$ implies that there are $k_j$ terms with $\gamma_{i_{j'}}\Gamma_K$ on the left hand side and $k_{j'}$ terms with $\gamma_{i_{j'}}\Gamma_K$ on the right hand side of $\gamma_{i_l}u=u.$ Hence, $k_j=k_{j'}$ and so (2) holds.   

(iii) If $R$ is a group ring, $\phi(\gamma)=\phi_\gamma$ is the identity on $F$ for any $\gamma\in\Gamma$ where $\phi$ and $\phi_\gamma$ are the maps from the proof of (ii2) $\Rightarrow$ (ii3). By the definition of $\phi,$ $\gamma_i\Gamma_K=\gamma^{-1}\gamma_{i}\Gamma_K$ for any $i=1,\ldots, n$ and any $\gamma\in\Gamma.$ Taking $\gamma$ to be $\gamma_j^{-1}$ for arbitrary $j=1,\ldots, n$ and $i$ to be such that $\gamma_i\Gamma_K=\Gamma_K,$ we have that $\Gamma_K=\gamma_j\Gamma_K.$ Since $j$ was arbitrary, $\Gamma=\Gamma_K.$ This shows (1) $\Rightarrow$ (2). 
Conversely, if $\Gamma=\Gamma_K,$ then $\gamma_i\Gamma_K=\Gamma_K=\gamma^{-1}\gamma_{i}\Gamma_K$ for any $i=1,\ldots, n$ and any $\gamma\in
\Gamma$ so $\phi_\gamma$ is the identity $1_F.$ Hence, $\{1_F\}$ is the image of $\phi$ and so (1) holds. 
The conditions (2), (3), and (4) are clearly equivalent.   
\end{proof}

By Theorem \ref{three_characterizations}, if $K$ is a trivially $\Gamma$-graded field and $\Gamma$ is infinite, then $R=\M_n(K)(\gamma_1, \dots,  \gamma_n)$ is not strongly graded for any $\gamma_1, \dots,  \gamma_n\in\Gamma.$ 

We consider another special case of Theorem \ref{three_characterizations} which is relevant for Leavitt path algebras. 
To shorten the notation, if each $\gamma_i\in \Gamma, i=1,\ldots,m,$ appears $k_i$ times in the list $$\gamma_1,\gamma_1,\ldots, \gamma_1,\; \gamma_2,\gamma_2\ldots, \gamma_2,\; \ldots\ldots\ldots, \;\gamma_m,\gamma_m,\ldots, \gamma_m,$$ we abbreviate this list as $$k_1(\gamma_1), k_2(\gamma_2),\ldots, k_m(\gamma_m).$$ So, if $K$ is a graded division ring, we use the following abbreviation  
\[\M_n(K)(\gamma_1,\gamma_1,\ldots, \gamma_1,\, \gamma_2,\gamma_2\ldots, \gamma_2,\, \ldots\ldots\ldots,\, \gamma_m,\gamma_m,\ldots, \gamma_m)=
\M_n(K)(k_1(\gamma_1), k_2(\gamma_2),\ldots, k_m(\gamma_m))\] For example, if $\Gamma$ is $\Zset,$ we shorten $\M_9(K)(0,0,0,0,1,1,1,2,2)$ as $\M_9(K)(4(0),3(1),2(2)).$

If $K$ is any field, we consider the algebra $K[x^m, x^{-m}]$ graded by $\Zset$ as in Section \ref{subsection_LPAs}. 

\begin{corollary}
Let $m$ and $n$ be positive integers and $\gamma_1, \ldots, \gamma_n$ be integers such that, when considered modulo $m$ and when arranged in a non-decreasing list, the list becomes $k_0(0), \ldots, k_{m-1}(m-1)$ for nonnegative integers  $k_0,\ldots, k_{m-1}.$ If $R=\M_n(K[x^m, x^{-m}])(\gamma_1, \ldots, \gamma_n),$ then the following hold. 
\begin{enumerate}[\upshape(1)]
\item $R$ is strongly graded if and only if $k_0, k_1,\ldots, k_{m-1}$ are positive. 

\item $R$ is a crossed product if and only if $k_0, k_1,\ldots, k_{m-1}$ are positive and equal to each other. 
In this case, $R$ is a skew group ring. 

\item $R$ is a group ring if and only if $m=1.$ In this case, $R\cong_{\gr}\M_n(K[x, x^{-1}])(0, 0, \ldots, 0).$
\end{enumerate}
\label{matrix_alg_corollary} 
\end{corollary}
\begin{proof}
Note that  $\Gamma/\Gamma_{K[x^m, x^{-m}]}=\Zset/m\Zset.$ If $R'=\M_n(K[x^m, x^{-m}])(k_0(0), k_1(1), \ldots, k_{m-1}(m-1)),$ then $R\cong_{\gr}R'$ by Lemma \ref{lemma_on_shifts}. The statements follow directly by applying Theorem \ref{three_characterizations} to $R'.$ The last sentence of part (3) follows by part (3) of Lemma \ref{lemma_on_shifts}. 
\end{proof}

\section{Crossed product and group ring Leavitt path algebras}\label{section_cross_product_LPAs}

In this section, we prove the main result, Theorem \ref{main}. Let (EDL) be the graph property below.  
\begin{itemize}
\item[(EDL)] For every cycle of length $m,$ the lengths, considered modulo $m$, of all paths which do not contain the cycle and which end in an arbitrary but fixed vertex of the cycle, are $k(0),\ldots, k(m-1)$ for some positive integer $k.$ 
\end{itemize}
The notation EDL shortens ``equally distributed lenghts''. Example \ref{example_cross_product} contains finite graphs which have (EDL) and a graph which fails to have (EDL).

\begin{theorem} Let $E$ be a graph and $K$ a field. In conditions (i5) and (ii4) and in the proof, $\Gamma$ is the infinite cyclic group generated by $x.$
\begin{enumerate}[\upshape(i)]
\item The following conditions are equivalent. 
\begin{enumerate}[\upshape(1)]
\item The algebra $L_K(E)$ is a crossed product.

\item The graph $E$ is a finite no-exit graph without sinks such that Condition (EDL) holds.  

\item $L_K(E)$ is graded isomorphic to an algebra of the form 
\[R=\bigoplus_{i=1}^n\M_{k_im_i}(K[x^{m_i}, x^{-m_i}])(k_i(0), k_i(1), \ldots, k_i(m_i-1))\]
where $n,$ $k_i$ and $m_i$ are positive integers for $i=1,\ldots, n.$   

\item The algebra $L_K(E)$ is a skew group ring. 

\item The set $E^0$ is finite and 
\[(G_E^\Gamma, 1_E)\cong \bigoplus_{i=1}^n \left(\Zset[x_i]/(x_i^{m_i}=1), \sum_{j=0}^{m_i-1}k_i x^j\right)\]
for some positive integers $n,$ $k_i$ and $m_i$ for $i=1,\ldots, n.$ 
\end{enumerate}

\item 
The following conditions are equivalent. 
\begin{enumerate}[\upshape(1)]
\item The algebra $L_K(E)$ is a group ring.

\item The graph $E$ is  a finite no-exit graph without sinks such that every cycle is of length one. 

\item $L_K(E)$ is graded isomorphic to an algebra of the form 
\[\bigoplus_{i=1}^n\M_{k_i}(K[x, x^{-1}])(k_i(0))\]
where $n$ and $k_i$ are positive integers for $i=1,\ldots, n.$   

\item  The set $E^0$ is finite  and $(G_E^\Gamma, 1_E)\cong \bigoplus_{i=1}^n (\Zset, k_i)$ for some positive integers $n$ and $k_i$ for $i=1,\ldots, n.$
\end{enumerate}
\end{enumerate}
\label{main} 
\end{theorem}
\begin{proof}
(i) We show (1) $\Rightarrow$ (2) first. If (1) holds, then $L_K(E)$ is unital and so $E^0$ is a finite set. By Theorem \ref{NvO_theorem}, $L_K(E)\cong_{\gr}(1)L_K(E)$ and so the relation $[L_K(E)]=x[L_K(E)]$ holds in $\V^\Gamma(L_K(E)).$ Considering the isomorphism $(\V^\Gamma(L_K(E)), [L_K(E)])\cong (M_E^\Gamma, 1_E),$ we obtain that $1_E=\sum_{v\in E^0}[v]\in M_E^\Gamma$ is such that $1_E=x1_E$ and so the element $\sum_{v\in E^0}[v]$ is periodic. Since a summand of a periodic element is periodic by \cite[Theorem 4.1]{Roozbeh_Lia_comparability}, $[v]$ is periodic for every vertex $v$ of $E.$ 

By Theorem \ref{periodic}, any path originating at any vertex of $E$ is a prefix of a path $p$ ending in one of finitely many cycles with no exits and such that all vertices of $p$ are regular. Since there are finitely many vertices, there are finitely many such cycles with no exits. The vertices of the paths leading to the cycles are regular and vertices on the cycles emit exactly one edge, so $E$ is a row-finite graph. A row-finite graph with finitely many vertices is finite, so $E$ is finite. 

If there is a cycle with an exit in $E,$ then $[v]$ is not periodic for every vertex $v$ of that cycle (see \cite[Lemma 3.8]{Roozbeh_Lia_comparability}). As this cannot happen, $E$ is a no-exit graph. By Theorem \ref{three_characterizations} and Lemma \ref{lemma_on_shifts}, if $R$ is a matricial representation of $L_K(E),$ then $R$ necessarily has a form as in part (3) where $n$ is the number of cycles, $m_i$ their lengths and the lengths, considered modulo $m_i$, of all paths which do not contain the cycle and which end in an arbitrary vertex of the cycle, are $k_i(0),\ldots, k_i(m_i-1)$ for some positive $k_i$  for all $i=1, \ldots, n.$ Hence, (2) holds. 

If $E$ is a graph as in (2), $n$ is the number of cycles in $E$, $m_i$ their lengths and $k_i$ the number as in Condition (EDL) for the $i$-th cycle, then a graded matricial representation of $L_K(E)$ has the form as in condition (3). So, (2) $\Rightarrow$ (3) holds. 

The algebra $R$ in (3) is a skew group ring by Corollary \ref{matrix_alg_corollary} so (3) $\Rightarrow$ (4). The implication (4) $\Rightarrow$ (1) is direct, so the conditions (1) to (4) are equivalent. Since the implication (3) $\Rightarrow$ (5) is also direct, to complete the proof of (i), we show that (5) $\Rightarrow$ (1).

If (5) holds, then $x\sum_{j=0}^{m_i-1}k_i x^j= \sum_{j=0}^{m_i-1}k_i x^j$ and so $x1_E=1_E$ holds in $G_E^\Gamma.$ Consequently, $[L_K(E)]=x[L_K(E)]$
holds in $K_0^\Gamma(L_K(E))$ and, hence, in $\V^\Gamma(L_K(E))$ also. Thus, $L_K(E)\cong_{\gr} (1)L_K(E)$ which implies that $L_K(E)\cong_{\gr} (n)L_K(E)$ for any integer $n.$ Hence, $L_K(E)$ is a crossed product by Theorem \ref{NvO_theorem}.  

(ii) To show (ii), we show (1) $\Rightarrow$ (2) $\Rightarrow$ (3) $\Rightarrow$ (1) and (3) $\Leftrightarrow$ (4). 

If (1) holds, then (i1) holds so (i2) and (i3) hold. By Theorem \ref{three_characterizations}, (1) and (i3) imply that $m_i=1$ for all $i=1,\ldots, n$ for $R$ as in (i3) and so all cycles of $E$ have length one. Hence, (2) holds.   

If $E$ is a graph as in (2), $n$ is the number of cycles in $E,$ and $k_i$ the number of paths to the vertex of the $i$-th cycle, then a graded matricial representation of $L_K(E)$ has the form as in condition (3). So, (2) $\Rightarrow$ (3) holds. The algebra in condition (3) is a group ring by Corollary \ref{matrix_alg_corollary} so (3) $\Rightarrow$ (1).

The implication  (3) $\Rightarrow$ (4) is direct. To show the converse, assume that (4) holds. Then $xa=a$ holds for every element $a\in G_E^\Gamma.$ Since (4) implies (i5) with $m_i=1$ for all  $i=1,\ldots, n,$ $L_K(E)\cong_{\gr} R$ for some graded algebra $R$ as in condition (i3) with $m_i=1$ for all $i=1,\ldots, k.$ Hence, (3) holds.   
\end{proof}

Theorem \ref{main} enables one to directly check whether a finite graph is a crossed product. 

\begin{example}
Consider the finite no-exit graphs without sinks below. 
$$
\xymatrix{ {\bullet} \ar[r]& {\bullet} \ar@/^1pc/ [r]   & {\bullet} \ar@/^1pc/ [l] }\hskip2cm
\xymatrix{ {\bullet} \ar[r]&{\bullet} \ar[r]& {\bullet} \ar@/^1pc/ [r]   & {\bullet} \ar@/^1pc/ [l]}  \hskip2cm 
\xymatrix{ {\bullet} \ar[r]& {\bullet} \ar@/^1pc/ [r]   & {\bullet} \ar@/^1pc/ [l] &\ar[l] {\bullet}}$$

For the first graph,  0, 1, and 1 are the lengths (modulo 2) of paths which end at any vertex of the cycle and which do not contain the cycle. Since the numbers of zeros and ones on this list are not equal, the Leavitt path algebra of this graph is not a crossed product.
For the second two graphs, the lengths in question (modulo 2) are 0, 0, 1, and 1. So, the Leavitt path algebras of the last two graphs are crossed products.   
\label{example_cross_product} 
\end{example}

\section{Strongly graded Leavitt path algebras}\label{section_strongly_graded}

By \cite[Theorem 3.15]{Roozbeh_Israeli} (or \cite[Theorem 1.6.13]{Roozbeh_book}), if $E$ is a finite graph, $L_K(E)$ is strongly graded if and only if $E$ has no sinks. This result was extended in \cite[Theorem 4.2]{Lisa_et_al_strongly_graded} by showing that $L_K(E)$ is strongly graded if and only if $E$ is a row-finite graph without sinks such that Condition (Y), given below, holds.  
\begin{itemize}
\item[(Y)] For every positive integer $k$ and every infinite path $p$ there exists an initial subpath $q$ of $p$ and a path $r$ such that $\ra(r) =\ra(q)$ and $|r|-|q|=k.$ 
\end{itemize}

Theorem \ref{strongly_graded} characterizes the condition that $L_K(E)$ is strongly graded in terms of the properties of $M_E^\Gamma$ (and $G_E^\Gamma$) in case when $E^0$ is finite. In the case that $E^0$ is finite, Theorem \ref{strongly_graded} also shows that it is not necessary to require Condition (Y) in order to have that $L_K(E)$ is strongly graded in \cite[Theorem 4.2]{Lisa_et_al_strongly_graded} and that it is not necessary to require that $E^1$ is finite for the implication $L_K(E)$ is strongly graded $\Rightarrow$ $E$ has no sinks in \cite[Theorem 1.6.13]{Roozbeh_book} or \cite[Proposition 45]{Nystedt_Oinert} (see also graphs in part (1) and (3) of Example \ref{example_strongly_graded}).

\begin{theorem} Let $E$ be a graph with finitely many vertices and let $\Gamma$ be the infinite cyclic group generated by $x$. The following conditions are equivalent. 
\begin{enumerate}
\item $L_K(E)$ is strongly graded.

\item $1_E$ is a strong order-unit of $M_E^\Gamma$ (equivalently of $G_E^\Gamma$). 

\item $E$ is a row-finite graph with no sinks. 
\end{enumerate}
\label{strongly_graded} 
\end{theorem}
\begin{proof}
By Proposition \ref{strong_order_unit_prop}, (1) $\Leftrightarrow$ (2). We show (2) $\Rightarrow$ (3) and (3) $\Rightarrow$ (2). 

If (2) holds, we show that $E$ has neither sinks nor infinite emitters. Assume that $v\in E^0$ is a sink. Let $n$ be a positive integer such that $x^{-1}[v]\leq n 1_E$ so that $x^{-1}v+a\sim n \sum_{v\in E^0}v$ holds in $F_E^\Gamma$ for some $a\in F_E^\Gamma.$ By Lemma \ref{confluence}, there is $c\in F_E^\Gamma$ such that $x^{-1}v+a\to c$ and $n \sum_{v\in E^0}v\to c.$ Since $v$ is a sink, none of (A1) to (A3) can be applied to $x^{-1}v$ and so $x^{-1}v$ is a summand of $c.$ However, since the axioms (A1) to (A3) increase the powers of $x$ or leave them the same in the resulting terms, the relation $n \sum_{v\in E^0}v\to c$ implies that no summand of $c$ can be of the form $x^{-1}v.$ Hence, we reach a contradiction.  

Assume that $v\in E^0$ is an infinite emitter. Following the steps of the previous case, we obtain that $x^{-1}v+a\to c$ and $n \sum_{v\in E^0}v\to c$ holds for some positive integer $n$ and some $a,c\in F_E^\Gamma.$ The first case shows that the term $x^{-1}v$ cannot remain unchanged in the process of obtaining $c$ from $x^{-1}v+a,$ so (A2) has to be used at some point for this term. Changing the order of the use of the axioms for $x^{-1}v+a\to c,$ we can assume that this application of (A2) is the first step. Hence, $x^{-1}q_Z+\sum_{w\in Z}\ra(e)+a\to c$ for some finite and nonempty $Z\subseteq \so^{-1}(v).$ In each subsequent step, either the term $x^{-1}q_Z$ is not changed, or (A3) is used for it and the result has a summand of the form $x^{-1}q_W$ for some finite and nonempty $W\subseteq\so^{-1}(v)$ such that $W\supsetneq Z.$ In any case, $c$ ends up having a summand of the form $x^{-1}q_Z$ for some finite and nonempty $Z\subseteq \so^{-1}(v).$ Using the same argument as in the first case, we can conclude that the relation $n \sum_{v\in E^0}v\to c$ implies that no summand of $c$ can be of the form $x^{-1}q_Z.$ Hence, we reach a contradiction in this case also.  

If (3) holds, then $E$ is a row-finite graph with finitely many vertices. Hence, $E$ is finite. Since $E$ is finite and has no sinks, every vertex connects to a cycle. Thus, every vertex of $E$ is either on a cycle or it connects to finitely many cycles.  

We claim that for any $v\in E^0,$ there is an element $a\in F_E^\Gamma-\{0\}$ with support containing only vertices on cycles such that $v\to a$. We prove this claim using induction on the minimum $n$ of the lengths of paths $p$ from $v$ to cycles such that only $\ra(p)$ is on a cycle. If $n=0,$ $v$ is on a cycle and one can take $a=v.$ Assuming the induction hypothesis, consider any $v\in E^0$ with $n>0.$ If $v$ is on a cycle, we can take $a=v$ again. If $v$ is not on a cycle, the vertex $v$ is regular since $E$ is finite and $v$ is not a sink, so $\so^{-1}(v)$ is nonempty and finite and (A1) can be applied to $v.$ For every $e\in \so^{-1}(v),$ the minimum of lengths of paths from $\ra(e)$ to cycles is less than $n$ and we can use the induction hypothesis to obtain $a_e\in F_E^\Gamma$ with vertices in the support on cycles and $\ra(e)\to a_e$. Then $a=\sum_{e\in  \so^{-1}(v)}xa_e$ has vertices in the support on cycles and \[v\to \sum_{e\in  \so^{-1}(v)}x\ra(e)\to \sum_{e\in  \so^{-1}(v)}xa_e=a.\]

This shows that, for $v\in E^0,$ one can find a positive integer $m,$ integers $k_i$ and vertices $v_i$ on cycles for $i=1,\ldots, m$ such that $v\to \sum_{i=1}^m x^{k_i}v_i$ holds in $F_E^\Gamma.$ Moreover, since the axioms either increase the degrees of monomials or leave them the same, $k_i\geq 0$ for all $i.$ If $n_i$ is the length of the cycle $c_i$ which contains $v_i,$ $w_i$ is a vertex on $c_i,$ and $t_i$ is a positive integer, then $w_i\to x^{t_in_i}w_i+b_i$ some $b_i\in F_E^\Gamma$ where $\to$ is obtained by applying the axioms (in this case only (A1) is applicable) following the cycle $c_i$ for $t_i$ revolutions\footnote{Note that $b_i=0$ if and only if the cycle on which $w_i$ is has no exits by \cite[Lemma 3.8]{Roozbeh_Lia_comparability}.}. Let us consider two cases $k_i+1\leq n_i$ and $k_i+1> n_i.$ If $k_i+1\leq n_i,$ let $w_i$ be the vertex on the cycle $c_i$ such that the length of the path on the cycle from $w_i$ to $v_i$ is exactly $k_i+1.$ Then, $w_i\to x^{k_i+1}v_i+d_i$ for some $d_i\in F_E^\Gamma$ by applying (A1) following that path. If $k_i+1> n_i,$ let $k_i+1=t_in_i+l_i$ for some positive $t_i$ and $0\leq l_i<n_i.$ Let $w_i$ be the vertex on the cycle $c_i$ such that the length of the path on the cycle from $w_i$ to $v_i$ is exactly $l_i.$ 
Then, \[w_i\to x^{t_in_i}w_i+b_i\to x^{t_in_i+l_i}v_i+d_i'+b_i= x^{k_i+1}v_i+d_i\] for some $b_i,d_i'\in F_E^\Gamma$ and $d_i=b_i+d_i'.$ Thus, in either case, the relation $[w_i]\geq x^{k_i+1}[v_i]$ holds in $M_E^\Gamma.$ Let 
$l_v$ be the maximal number of repetitions of a term in the sum  $\sum_{i=1}^m [w_i],$ $k_v$ be the maximal number of repetitions of a term in the sum $\sum_{e\in \so^{-1}(v)}[\ra(e)]$ and  $l$ be the maximum of the set $\{l_v\, |\, v\in E^0\}\cup\{k_v\,|\, v\in E^0\}.$
Thus, we have that
\[x[v]=\sum_{i=1}^m x^{k_i+1}[v_i]\leq \sum_{i=1}^m [w_i]\leq \sum_{w\in E^0}l_v[w]=l_v1_E\leq l1_E\]
which implies that \[x1_E=\sum_{v\in E^0}x[v]\leq \sum_{v\in E^0} l1_E=|E_0|l1_E\]
and so for any $v\in E^0,$ $x^2[v]\leq xl1_E\leq |E_0|l^21_E.$ Continuing this argument, we obtain that \[x^n[v]\leq |E_0|^{n-1}l^n1_E.\]

In addition, since $v\to \sum_{e\in \so^{-1}(v)}x\ra(e)$ for any $v\in E^0,$ we have that  
\[x^{-1}[v]=\sum_{e\in \so^{-1}(v)}[\ra(e)]\leq \sum_{w\in E^0}k_v[w]=k_v1_E\leq l1_E\]
which implies that 
\[x^{-1}1_E=\sum_{v\in E^0}x^{-1}[v]\leq \sum_{v\in E^0} l1_E=|E_0|l1_E\]
and so for any $v\in E^0,$ $x^{-2}[v]\leq x^{-1}l1_E\leq |E_0|l^21_E.$
Continuing this argument, we obtain that \[x^{-n}[v]\leq |E_0|^{n-1}l^n1_E.\]

Thus, for any $a\in F_E^\Gamma$ with a normal representation $a=\sum_{j=1}^k x^{m_j}v_j,$  
\[[a]=\sum_{j=1, m_j=0}^k x^{m_j}[v_j]+\sum_{j=1, m_j\neq 0}^k x^{m_j}[v_j]\leq  \left(k+\sum_{j=1, m_j\neq 0}^k|E_0|^{|m_j|-1}\,l^{|m_j|} \right)1_E.\] Hence, $1_E$ is a strong order-unit and so (2) holds. 
\end{proof}

If every infinite path of $E$ ends in a cycle, then one can directly check that Condition (Y) holds. So, if $E$ is a row-finite graph without sinks such that every infinite path ends in a cycle, then $L_K(E)$ is strongly graded. The first graph below shows that the converse fails.  

\begin{example}
\begin{enumerate}
\item If $E$ is the row-finite graph without sinks below, $$\xymatrix{& \ar@{.>}[r] & \bullet \ar[r] & \bullet \ar[r] & \bullet\ar[r] & \bullet \ar@{.>}[r] &}$$ it is direct to check that Condition (Y) holds because of the existence of paths of all lengths ending at arbitrary vertex of any infinite path. Hence, $L_K(E)$ is strongly graded. However, $E$ has infinite paths which do not end in cycles.  

Since every infinite path of this graph is an infinite sink, this example also shows that, as opposed to finite sinks, the infinite sinks do not prevent the algebra $L_K(E)$ from being strongly graded. One may suspect that if $E$ is a row-finite graph without sinks such that every infinite path ends in a cycle {\em or} an infinite sink, then $L_K(E)$ is necessarily strongly graded. However, that is also not true as the next example shows.  

\item Let $E$ be the graph below. $$\xymatrix{\bullet \ar[r] & \bullet \ar[r] & \bullet \ar[r] & \bullet \ar@{.>}[r]&}$$
Condition (Y) fails for the infinite path containing every edge of this graph and $k=1$ since for any initial path $q$ of length $l$ there is no path of length $l+1$ which ends in $\ra(q).$ Hence, $L_K(E)$ is not strongly graded.

\item Based on the first example, one may question whether $L_K(E)$ being strongly graded implies that $E$ is a row-finite graph without sinks such that every infinite path ends in a cycle or an infinite sink. A consideration of the graph $E$ below shows that this is also not true.  

$$
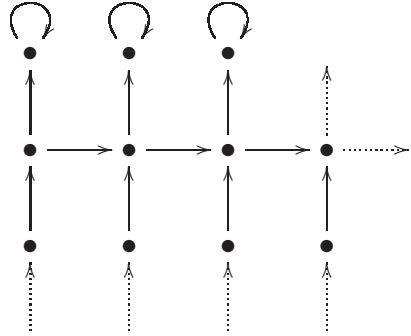
\xymatrix{   
\bullet  \ar@(ul,ur)   & \bullet \ar@(ul,ur)    & \bullet  \ar@(ul,ur) &  & \\   
\bullet \ar[r]\ar[u] & \bullet \ar[r]  \ar[u] & \bullet \ar[r]\ar[u] & \bullet \ar@{.>}[r] \ar@{.>}[u] &\\
\bullet \ar[u] & \bullet \ar[u] & \bullet\ar[u] & \bullet\ar[u]  &\\  
\ar@{.>}[u]    &  \ar@{.>}[u]   & \ar@{.>}[u]   & \ar@{.>}[u]   &
}$$
Just as the first graph, $E$ is a row-finite graph without sinks and it is direct to check that Condition (Y) holds. So, $L_K(E)$ is strongly graded. However, $E$ has infinite paths which end neither in cycles nor in infinite sinks.      
\end{enumerate}
\label{example_strongly_graded}
\end{example}

Next, we show that the converse of Proposition \ref{ur_and_sg_implies_cross_product} holds for unital Leavitt path algebras. 

\begin{corollary}
If $E$ is a graph with finitely many vertices, then the following are equivalent. 
\begin{enumerate}[\upshape(1)]
\item $L_K(E)$ is a crossed product. 
 
\item $E$ is a finite no-exit graph without sinks such that Condition (EDL) holds.  

\item $L_K(E)$ is strongly graded and graded unit-regular.  
\end{enumerate}
\label{unit_reg_corollary} 
\end{corollary}
\begin{proof}
By Theorem \ref{main}, (1) $\Leftrightarrow$ (2) and, by Proposition \ref{ur_and_sg_implies_cross_product}, (3) $\Rightarrow$ (1). We show (1) $\Rightarrow$ (3). 

If (1) holds, then $L_K(E)$ is strongly graded. By Theorem \ref{main}, $E$ is finite. So, the zero component $L_K(E)_0$ is an ultramatricial algebra over $K$ with unital connecting maps (see \cite[Proposition 2.1.14]{LPA_book}) and, hence, $L_K(E)_0$ is unit-regular. By Theorem \ref{main} also, $L_K(E)$ is a skew group ring over $L_K(E)_0$ and so $L_K(E)\cong_{\gr} L_K(E)_0*_{\phi}\Zset$ for some map $\phi.$ It is direct to check that the unit-regularity of $L_K(E)_0$ implies the graded unit-regularity of $L_K(E)_0*_{\phi}\Zset.$ So, $L_K(E)$ is graded unit-regular. 
\end{proof}

By \cite[Theorem 5.3]{Lia_cancellation}, if $E$ is a finite graph, then $L_K(E)$ is graded unit-regular if and only if $E$ is a no-exit graph without sinks which receive edges and such that Condition (EDL) holds. Thus, if condition (2) of Corollary \ref{unit_reg_corollary} holds then $L_K(E)$ is strongly graded by Theorem \ref{strongly_graded} and graded unit-regular by \cite[Theorem 5.3]{Lia_cancellation} so (3) holds. This provides an alternative proof of Corollary \ref{unit_reg_corollary}. 

Corollary \ref{unit_reg_corollary} and \cite[Theorem 5.3]{Lia_cancellation} enable us to easily create examples of strongly graded rings which are not graded unit-regular and graded unit-regular rings which are not strongly graded. Indeed, if $E$ is the first graph from Example \ref{example_cross_product}, then $L_K(E)$ is strongly graded since $E$ is a finite graph without sinks. However, Condition (EDL) fails so $L_K(E)$ is not graded unit-regular. On the other hand, if $E$ is a single vertex $\bullet$ with no edges, then $L_K(E)$ is unit-regular by \cite[Theorem 5.3]{Lia_cancellation} but $L_K(E)$ is not strongly graded since this vertex is a sink.

\end{document}